\documentclass[12pt]{article}
\usepackage{e-jc}

\usepackage[latin1]{inputenc}
\usepackage{subfigure}

\usepackage{amsmath}
\usepackage{amsthm}
\usepackage{mathrsfs}
\usepackage{amsfonts}
\usepackage{verbatim}
\usepackage{pgf, tikz}

\newtheorem{theorem}{Theorem}[section]

\newtheorem{proposition}[theorem]{Proposition}

\newtheorem{lemma}[theorem]{Lemma}

\newtheorem{definition}[theorem]{Definition}

\newtheorem{remark}[theorem]{Remark}

\newtheorem{example}[theorem]{Example}

\newcommand{\lk}{\mbox{lk}}

\newcommand{\del}{\mbox{del}}

\author{Jacob A. White}
\title{Decision trees, monotone functions, and semimatroids}

\begin{document}

\maketitle

\begin{abstract}
We define decision trees for monotone functions on a simplicial complex. We define homology decidability of monotone functions, and show that various 
monotone functions related to semimatroids are homology decidable. Homology decidability is a generalization of semi-nonevasiveness, a 
notion due to Jonsson. The motivating example is the complex of bipartite graphs, whose Betti numbers are unknown in general.
 We show that these monotone functions have optimum decision trees, from which we can compute relative Betti numbers of related pairs of simplicial complexes. 
Moreover, these relative Betti numbers are coefficients of evaluations of the Tutte polynomial, and every semimatroid collapses onto its broken circuit complex.
\end{abstract}

\section{Introduction}

Let $f: 2^S \to \{0,1 \}$ be a function. We wish to find a decision tree for computing $f$, and to minimize the depth of this decision tree, which is always at most $|S|$. 
In the current paper, we consider decision trees for monotone functions $f: \Delta \to P$, for some simplicial complex $\Delta$ and some poset $P$. 
We extend the notion of evasiveness to this setting. 
We also consider the problem of finding an \emph{optimal} decision tree when $f$ is monotone increasing. 
In this case, an \emph{optimal} decision tree gives information about relative homology groups
$\widetilde{H}_i(\Delta_{\leq p}, \Delta_{< p})$, where $\Delta_{\leq p}$ are all the faces whose function value is at most $p$, and $\Delta_{< p}$ 
The philosophy of this paper is to extend combinatorial methods in the study of the topology of simplicial complexes, by replacing simplicial complexes with 
monotone functions $f: \Delta \to P$.

There are numerous papers involving cases where $f$ is a monotone graph property, such as being a disconnected graph \cite{vassiliev}, or being bipartite \cite{chari}.
A well-known fact is that nonevasive complexes are contractible. So several graph properties 
have been proved to be evasive by showing that $\Delta$ has nontrivial homology in some dimension. 
We refer the interested reader to the book \emph{Simplicial complexes of Graphs} \cite{jonsson}, 
by Jonsson, which provides an extensive survey of several such complexes, and what is known about their topology.
When the graph property is the property of being acyclic \cite{billera-provan}, 
bipartite \cite{chari}, or not connected \cite{vassiliev}, the complex is actually known to be homotopy equivalent to a wedge of spheres. 

Robin Forman \cite{formanmorse} noted that any decision tree can be used to give inequalities on the Betti numbers of $\Delta$. 
When the inequalities are actually equalities, then the complex is called \emph{semi-nonevasive} \cite{jonsson-decision}. 
Jonsson showed that the properties of being bipartite, disconnected, or acyclic are all semi-nonevasive \cite{jonsson}. 
Moreover, he showed that this same result holds for a large 
family of simplicial complexes which he calls \emph{pseudo-independence complexes}. 
However, no formulas are given for the Betti numbers at this level of generality, and computing these Betti numbers was the primary motivation for this paper.

As shown in Section 5, there is a bijection between strong pseudo-independence complexes and semimatroids.
Semimatroids were originally introduced by Ardila \cite{ardila}, who studied a Tutte polynomial for semimatroids.
We show that the only non-zero Betti number of a semimatroid is the constant term of the Tutte polynomial. 
For hyperplane arrangements, 
it is known that this constant term also gives the only nonzero Betti number of the broken circuit complex (see \cite{orlik-welker}, Section 1.5, for details and references).
In fact, every semimatroid collapses onto its broken circuit complex, and it was attempting to understand the collapses that led us to consider the notion of homological 
decidability for monotone functions, which is a generalization of semi-nonevasiveness, and is introduced in the next section.
We show that several functions on semimatroids, including the rank function and nullity function, are homology decidable, 
and that the relative Betti numbers of these functions can be computed from the Tutte polynomial.
Thus we obtain a topological interpretation for certain evaluations of the Tutte polynomial.

The paper is organized as follows: first, we study decision trees and homology decidability for monotone functions. This section includes Lemmas \ref{lem:morseinequalities} and \ref{lem:simple}, 
which are proven in Section 6.
In Section 3, we review semimatroids, 
and prove that the rank and nullity functions on semimatroids are homology decidable. In Section 4, we relate our results to the study of broken circuits. 
In Section 5, we show the equivalence between strong pseudo-independence complexes and semimatroids. 
In Section 6, we present a discrete Morse theory for monotone functions, which we use to prove Lemmas \ref{lem:morseinequalities} and \ref{lem:simple}. 
In Section 7, we conclude with some open problems.
In Sections 4 and 6, we assume the reader is familiar with vertex decomposability (due to Provan and Billera \cite{billera-provan}) and the theory of simplicial collapse. 
Vertex decomposability is Definition 13.29 in \cite{kozlov}, and 
cellular collapse is Definition 11.12 in \cite{kozlov}.

\section{Homological decidability for monotone functions}

\begin{definition}
An abstract simplicial complex $\Delta$ over a set $S$ is a collection of subsets 
of $S$ such that, if $\sigma \in \Delta$ and $\tau \subseteq \sigma$, then $\tau \in \Delta$. Each $S \in \Delta$ is a \emph{face}, and 
faces of size one are vertices.
\end{definition}
We do not assume that every element of $S$ is a vertex of $\Delta$.
Throughout this section, let $P$ be a fixed poset.
A function $f: \Delta \to P$ is \emph{monotone} if $f(\sigma) \leq f(\tau)$ whenever $\sigma \subseteq \tau \in \Delta$.
Since $\Delta$ is partially ordered by containment, $f$ is monotone if and only if it is order-preserving.
An example of a simplicial complex $\Delta$, a poset $P$, and a monotone function $f$ are given in Figure~\ref{fig:example}.

\begin{figure}
\begin{center}
\includegraphics[height=3cm]{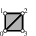}
\includegraphics[height=3cm]{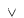}
\includegraphics[height=3cm]{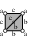}
\caption{A simplicial complex, a poset, and a monotone function (where $f(\emptyset) = a$)}
\label{fig:example}
\end{center}
\end{figure}

We extend the definition of link and deletion for a monotone function $f$. Given $\sigma \in \Delta$, recall that
$\lk_{\sigma}(\Delta) = \{\tau: \tau \cap \sigma = \emptyset, \tau \cup \sigma \in \Delta \}$ and $\del_{\sigma}(\Delta) = \{\tau \in \Delta: \sigma \not\subset \tau \}$ 
are the link and deletion of $\Delta$ with respect to $\sigma$. 
Define $f_{/ \sigma}: \lk_{\sigma}(\Delta) \to P$ by $f_{/ \sigma}(\tau) = f(\sigma \cup \tau)$. 
Similarly, define $f_{\setminus \sigma}: \del_{\sigma}(\Delta) \to P$ by $f_{\setminus \sigma}(\tau) = f(\tau)$. 
These define the deletion and link of $f$ with respect to $\sigma$.

Let $\Delta$ be a simplicial complex over $S$, and $f: \Delta \to P$ be a monotone function. 
An $f$-tree is a tree whose interior vertices are labeled with subsets of $S$, and whose leaves are 
labeled with elements of $P$, or with the symbol $N$ (corresponding to something not being in $\Delta$).

\begin{definition}
 Let $f: \Delta \to P$ be a monotone function. An $f$-tree $T$ is an element decision tree if it satisfies exactly one of the following:
\begin{enumerate}
 \item $T$ consists of only a single vertex, labeled by some $p \in P$, and $f$ is a constant function, whose value is $p$.
 \item The root of $T$ is labeled by an element $x \in S$, with $\{x \} \not\in \Delta$, the left subtree is a leaf labeled $N$, and the 
right subtree is an element decision tree for $f_{\setminus x}$.
 \item The root of $T$ is labeled by an element $x \in S$, with $\{x \} \in \Delta$, the left subtree $T_{\ell}$ is an element decision tree for $f_{/ x}$, 
and the right subtree $T_{r}$ is an element decision tree for $f_{\setminus x}$. 
\end{enumerate}
\end{definition}
By abuse of notation, we refer to an element decision tree as a decision tree.  
This definition is similar to the version given by Jonsson \cite{jonsson-decision}. 
Figure~\ref{fig:dectree} shows an example of a decision tree from the function given in Figure~\ref{fig:example}.

\begin{figure}
\begin{center}
\label{fig:dectree}
\includegraphics[height=5cm]{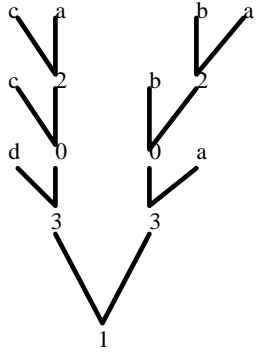}
\caption{A decision tree for $f$ in Figure~\ref{fig:example}}
\end{center}
\end{figure}

Let $T$ be an $f$-tree. Consider the following algorithm: given any set $\sigma$, 
query whether or not $x \in \sigma$, and recurse to right child of the root of $T$ if $x \in \sigma$. Otherwise, recurse to the left child. 
Upon reaching a leaf, return the value of the leaf. 
$T$ is a decision tree if and only if for every $\sigma$, either the returned value is $N$, and $\sigma \not\in \Delta$, or the returned value is $f(\sigma)$.

Given a decision tree $T$, and a face $\sigma \in \Delta$, $\sigma$ is $f(\sigma)$-evasive if we have to make $|S|$ queries 
before returning $f(\sigma)$.
For the decision tree in Figure~\ref{fig:dectree}, there are two $a$-evasive faces, one $b$-evasive face, and one $c$-evasive face. One could ask if there is a better 
decision tree, with no evasive faces.
Let $i \in \mathbb{N}$, and define $ev_{T}(f, p, i)$ to be the number of $p$-evasive faces $\sigma \in \Delta$ 
with $|\sigma| = i$. 
A monotone function $f$ is \emph{nonevasive} if it has no $p$-evasive faces, for all $p \in P$.

The next lemma is originally due to Forman \cite{formanmorse}, in the case where $f$ is constant. 
Given $p \in P$, let $\Delta_{\leq p} = \{ \sigma \in \Delta: f(\sigma) \leq p \}$ and define $\Delta_{< p}$ similarly. 
We also denote these complexes by $\Delta_{\leq p}^f$ and $\Delta_{< p}^f$ when there is some possibility for confusion.
\begin{lemma}[Morse Inequalities]
 \label{lem:morseinequalities}
Let $f: \Delta \to P$ be a monotone function, and let $T$ be a decision tree for $f$. 
Then for all $p \in P$ and $i \in \mathbb{N}$, we have 
\begin{equation}
 \label{eq:morseinequalities}
ev_{T}(f,p,i) \geq \widetilde{\beta}_{i+1}(\Delta_{\leq p}, \Delta_{< p})
\end{equation}
\end{lemma}
In particular, if at least one of the relative homology groups for one of the pairs $(\Delta_{\leq p}, \Delta_{< p})$ is non-trivial, then 
$f$ is evasive. For the monotone function appearing in Figure~\ref{fig:example}, $\Delta_{\leq a}$ is homotopy equivalent to 3 disjoint points, 
and so $\widetilde{\beta}_0(\Delta_{\leq a}) = 2$. Thus,
$f$ is evasive. Moreover, $\widetilde{\beta}_0(\Delta_{\leq b}, \Delta_{< b}) = 1 = \widetilde{\beta}_0(\Delta_{\leq c}, \Delta_{< c})$, 
so the Morse inequalities are exact.

\begin{definition}
  A decision tree $T$ for a monotone function $f: \Delta \to P$ is \emph{optimal} if equality holds in \eqref{eq:morseinequalities} for all $p \in P$, $i \in \mathbb{N}$. 
If $f$ has an optimal decision tree, then $f$ is \emph{homology decidable}.
\end{definition}
There is a recursive definition for homology decidable:
\begin{proposition}
Let $f: \Delta \to P$ be monotone. Then $f$ is homology decidable if and only if it satisfies exactly one of the following conditions:
\begin{enumerate}
\item $f$ is constant, and $\Delta$ is a simplex,
\item there exists $x \in S$ with $\{x\} \not\in \Delta$, and $f_{\setminus x}$ is homology decidable.
\item there exists a vertex $x$ (called a shedding vertex) such that $f_{\setminus x}$ and $f_{/x}$ are both homology decidable, and
$$\widetilde{H}_d(\Delta_{\leq p}^f, \Delta_{< p}^f) \simeq \widetilde{H}_d(\Delta_{\leq p}^{f_{\setminus x}},\Delta_{< p}^{f_{\setminus x}}) \oplus \widetilde{H}_{d-1}(\Delta_{\leq p}^{f_{/x}}, \Delta_{< p}^{f_{/x}})$$ for all $d \geq -1$, $p \in P$.
\end{enumerate}
\end{proposition}
When $f$ is constant, we recover the notion of semi-nonevasiveness due to Jonsson \cite{jonsson-decision}. 
as shown in Example \ref{ex:badexample}, there are monotone functions $f: \Delta \to P$ which have optimal decision trees, but for which every face of $\Delta$ is $p$-evasive for some $p \in P$. Motivated by this example, and 
the fact that the Betti numbers are computed from a decision tree, we have chosen the name `homology decidability'.

There is a sufficient condition for a function to be homology decidable.
The proof relies on the Fundamental Theorem of Discrete Morse Theory for monotone functions, and is given in Section~\ref{sec:discmorse}. 
\begin{lemma}
\label{lem:simple}
Let $f: \Delta \to P$ be monotone, with a decision tree $T$. 
Suppose there exists a function $d: P \to \mathbb{N}$ such that, 
if $\sigma$ is $p$-evasive, then $|\sigma| = d(\sigma)$.
Then $T$ is optimal, and $f$ is homology decidable.
\end{lemma}
In other words, the dimension of an evasive face $\sigma$ depends only on $f(\sigma)$.

\begin{example}
\label{ex:badexample}
Let $\Delta$ be a simplicial complex, and define $f: \Delta \to \mathbb{N}$ by $f(\sigma) = |\sigma|$. The only way to determine $f(\sigma)$ is to query every vertex of $\Delta$, so 
every face is evasive. Applying Lemma~\ref{lem:simple} to $d = f$, we see that $f$ is homology decidable.
\end{example}

We end this section by defining the \emph{Poincar\'e polynomial} for a monotone function $f: \Delta \to \mathbb{N}$. It is given by 
\begin{equation}
 \label{eq:poincare}
P(f; q,x) = \sum_{i \geq 0, j \geq 0} \widetilde{\beta}_i(\Delta_{\leq j}, \Delta_{< j}) x^jq^i 
\end{equation}

\section{Semimatroids}
Semimatroids were originally introduced by Ardila \cite{ardila}, and form a generalization of affine hyperplane arrangements. Ardila studied 
the Tutte polynomial for semimatroids, and computed them for various affine hyperplane arrangements \cite{ardila-tutte}. The goal of this section is to show that many monotone functions on
semimatroids are semi-nonevasive, and their Poincar\'e polynomials are given by evaluations of the Tutte polynomial. 
The background material in this section, including Theorem 3.4, comes from \cite{ardila}. Our original contribution is Theorem \ref{thm:main}.

\begin{definition} A \emph{semimatroid} is a triple $\mathcal{C} = (S, \Delta, r)$ where $\Delta$ is a non-void simplicial complex over $S$, 
and $r: \Delta \to \mathbb{N}$ such that we have the following, for any $\sigma, \tau \in \Delta$:
\begin{enumerate}
\item $0 \leq r(\sigma) \leq |\sigma|$.
\item If $\sigma \subseteq \tau$, then $r(\sigma) \leq r(\tau)$.
\item If $\sigma \cup \tau \in \Delta$, then $r(\sigma) + r(\tau) \geq r(\sigma \cup \tau) + r(\sigma \cap \tau)$.
\item If $r(\sigma) = r(\sigma \cap \tau)$, then $\sigma \cup \tau \in \Delta$.
\item If $r(\sigma) < r(\tau)$, then there exists $y \in \tau \setminus \sigma$ such that $\sigma \cup \{y \} \in \Delta$.
\end{enumerate}
\end{definition}

\begin{example}
Let $\mathcal{H}$ be a collection of affine hyperplanes in $\mathbb{R}^n$. Then $\{H_1, \ldots, H_k \}$ is intersecting if $\cap_{i=1}^k H_i \neq \emptyset$. 
Let $\Delta(\mathcal{H})$ be the collection of all intersecting sets of hyperplanes. Given an intersecting set $\sigma$, define $r(\sigma) = n - \dim(\cap \sigma)$, the codimension of the intersection of the hyperplanes. 
Then $(\mathcal{H}, \Delta(\mathcal{H}), r)$ is a semimatroid.
\end{example}

Now we recall the definition of Tutte polynomial of a semimatroid $\mathcal{C}$.
A set $X \in \Delta$ is \emph{dependent} if $r(X) < |X|$, and \emph{independent} otherwise. A \emph{circuit} is a minimal dependent set. A maximal independent set is a basis, and every basis has the same size, which is 
denoted $r_{\mathcal{C}}$, the rank of the semimatroid. A circuit of size $1$ is called a \emph{loop}, and an element of $S$ that is in every basis is called a \emph{coloop}.

\begin{definition}
 
The Tutte polynomial of a semimatroid $\mathcal{C}$ is the polynomial $$T_{\mathcal{C}}(x,y) = \sum_{\sigma \in \Delta} (x-1)^{r_{\mathcal{C}} - r(\sigma)} (y-1)^{|\sigma| - r(\sigma)}$$
\end{definition}

There is also a deletion-contraction recurrence for $T_{\mathcal{C}}$, due to Ardila \cite{ardila}. Given $e \in \mathcal{C} = (S, \Delta, r)$, define $\mathcal{C} - e = (S \setminus \{e\}, \del_{e}(\Delta), r_{\setminus e})$ 
to be the deletion. Similarly, define $\mathcal{C}/e = (S \setminus \{e \}, \lk_{e}(\Delta), r')$ where $r'(\sigma) = r(\sigma + e) - r(e)$.  
\begin{theorem}[\cite{ardila}, Proposition 8.2]
 
Let $\mathcal{C}$ be a semimatroid, let $e \in S$. Then $T_{\mathcal{C}}$ satisfies the following recurrence:
\begin{enumerate}
\item $T_{\emptyset} = 1$,
 \item $T_{\mathcal{C}} = T_{\mathcal{C} - e}$ if $\{e \} \notin \Delta$,
\item $T_{\mathcal{C}} = yT_{\mathcal{C} - e}$ if $e$ is a loop,
\item $T_{\mathcal{C}} = xT_{\mathcal{C} / e}$ if $e$ is a coloop,
\item $T_{\mathcal{C}} = T_{\mathcal{C} - e} + T_{\mathcal{C} / e}$ otherwise.
\end{enumerate}

\end{theorem}

Our main theorem is the following:
\begin{theorem}
\label{thm:main}
Let $\mathcal{C} = (S, \Delta, r)$ be a semimatroid, with $S \neq \emptyset$. 
\begin{enumerate}
\item Let $r: \Delta \to \mathbb{N}$ be the rank function. Then $r$ is homology decidable, and $P(r; q,x) = (qx)^{r_{\mathcal{C}}}T_{\mathcal{C}}(\frac{qx+1}{qx}, 0) $.
\item Let $n: \Delta \to \mathbb{N}$ be given by $n(\sigma) = |\sigma| - r(\sigma)$. This is the nullity function. Then $n$ is monotone and homology decidable. Moreover, $P(n; q,x) = (q)^{r_{\mathcal{C}}}T_{\mathcal{C}}(0, qx+1)$.
\item $\Delta$ is semi-nonevasive, and $P(\Delta; q) = T_{\mathcal{C}}(0,0)q^{r_{\mathcal{C}}}$.
\end{enumerate}

\end{theorem}

The proof relies on using induction and constructing decision trees recursively. Part of the proof involves computing the number of 
$i$-evasive faces for each $i$ in a decision tree. Given a decision tree $T$ for $f$, define 
\begin{equation}
\label{eq:evasive}
 E(T, f; q, x) = \sum_{i \geq 0, j \geq 0} |Ev_{T,j,i}| x^j q^i 
\end{equation}
When $T$ is optimal, we have $E(T, f; q,x) = P(f;q,x)$.

Now we mention constructions on decision trees. 
Let $f: \Delta \to P$ be a monotone function, and let $x \in S$. Suppose we have decision trees 
$T_{\setminus x}$ for $f_{\setminus x}$ and $T_{/ x}$ for $f_{/x}$. Then $T = T_{/x} \vee_{x} T_{\setminus x}$ is the decision tree with root labeled $x$, whose left child 
is the root of $T_{/x}$, and whose right child is the root of $T_{\setminus x}$. Clearly $T$ is a decision tree for $f$.
Also, if we have a decision tree $T$ for $f: \Delta \to \mathbb{N}$, let $T^+$ be obtained by incrementing all the leaf labels by one. Then $T^+$ is a decision tree for $f^+$, 
where $f^+(\sigma) = f(\sigma) + 1$.

\begin{proof}
First, we show that the rank function is homology decidable. We prove, by induction, that there exists a decision tree $D$ for $\mathcal{C}$ such that 
$E(D, r_{\mathcal{C}}; q,x) = (qx)^{r_{\mathcal{C}}}T_{\mathcal{C}}(\frac{qx+1}{qx}, 0)$. 
The expression on the right hand side is a polynomial in $qx$, so $Ev_{D,i,j} = \emptyset$ whenever $i \neq j$. 
By applying Lemma~\ref{lem:simple}, with $d$ being the identity function, we conclude that $D$ is optimal, $r_{\mathcal{C}}$ is 
homology decidable, and $P(r_{\mathcal{C}}; q,x) = E(D, r_{\mathcal{C}}; q, x)$.

For notational convenience, define $P(\mathcal{C}) = (qx)^{r_{\mathcal{C}}}T_{\mathcal{C}}(\frac{qx+1}{qx}, 0)$.
The leaf labeled $0$ is an optimal decision tree for $r_{\emptyset}$. So suppose $|S| > 0$, and let $e \in S$.
By induction, there exists optimal decision trees $D_{/e}$ for $r_{\mathcal{C}/e}$ and $D_{\setminus e}$ for $r_{\mathcal{C}-e}$.
The proof breaks up into cases:
\begin{enumerate}
 \item Suppose $e$ is a loop. Define $D = D_{\setminus e}$. This is a decision tree for $r_{\mathcal{C}}$, because $r(\sigma-e) = r(\sigma+e)$ for all $\sigma \in \Delta$. 
Thus $r_{\mathcal{C}}$ is nonevasive, and $E(D, r_{\mathcal{C}}; q,x) = 0 = P(\mathcal{C})$.

 \item Suppose $e \not\in \Delta_{\mathcal{C}}$. Define $D = N \vee_e D_{\setminus e}$, where $N$ is the tree with one vertex, labeled $N$. This is a decision tree for 
$r_{\mathcal{C}}$ and
 \[E(D, r_{\mathcal{C}}; q,x) = E(D_{\setminus e}, r_{\mathcal{C}-e};q,x) = P(\mathcal{C}-e) = P(\mathcal{C}) \]

 \item Suppose $e$ is a coloop. Define $D = D_{/e}^+ \vee_e D_{/e}$. This is a decision tree for 
$r_{\mathcal{C}}$. For $\sigma \in \del_e(\Delta)$, this follows because $r_{\mathcal{C}/e} = r_{\mathcal{C} - e}$, so $D_{/e}$ is 
a decision tree for $r_{\mathcal{C}-e}$. 
For $\sigma \in \lk_e(\Delta)$, note that $r_{\mathcal{C}}(\sigma) = r_{\mathcal{C}/e}(\sigma)+ 1$, so $D_{/e}^+$ is a decision tree for $r_{\mathcal{C}}$ restricted to $\lk_e(\Delta)$.
Thus \[E(D, r_{\mathcal{C}}; q,x) = (1+qx) E(D_{/e}, r_{\mathcal{C}/e}; q,x) = qx(\frac{1+qx}{qx})P(\mathcal{C}/e) = P(\mathcal{C}) \]

 \item Suppose $e$ is not a loop nor coloop, and $e \in \Delta_{\mathcal{C}}$. Define $D = D_{/e}^+ \vee_e D_{\setminus e}$. This is a decision tree for 
$r_{\mathcal{C}}$. For $\sigma \in \del_e(\Delta)$ this is clear. For $\sigma \in \lk_e(\Delta)$, note that $r_{\mathcal{C}}(\sigma) = r_{\mathcal{C}/e}(\sigma)+1$, so $D_{/e}^+$ is a decision tree for $r_{\mathcal{C}}$ restricted to $\lk_e(\Delta)$.
Thus \[E(D, r_{\mathcal{C}}; q,x) = qx E(D_{/e}, r_{\mathcal{C}/e}; q,x) + E(D_{\setminus e}, r_{\mathcal{C}-e};q,x) = qxP(\mathcal{C}/e) + P(\mathcal{C}-e) = P(\mathcal{C}) \]
\end{enumerate}

Now we show that the nullity function $n$ is homology decidable. This time, set $Q(\mathcal{C}) = (q)^{r_{\mathcal{C}}}T_{\mathcal{C}}(0, qx+1)$. We construct a decision tree $D_{\mathcal{C}}$ such that $E(D_{\mathcal{C}}, n; q,x) = Q(\mathcal{C})$. 
Every term in $Q(\mathcal{C})$ is of the form $q^{r_{\mathcal{C}}}$ times $(qx)^j$ for some $j$. The result follows by applying Lemma~\ref{lem:simple}, with the function $d$ defined by $d(x) = x + r_{\mathcal{C}}$.
The leaf labeled $0$ is an optimal decision tree for $n_{\emptyset}$. So suppose $|S| > 0$, and let $e \in S$. 
By induction, there exists optimal decision trees $D_{/e}$ for $n_{\mathcal{C}/e}$ and $D_{\setminus e}$ for $n_{\mathcal{C}-e}$.
The proof breaks up into cases:
\begin{enumerate}
 \item Suppose $e$ is a loop. Define $D = D_{\setminus e}^+ \vee_e D_{\setminus e}$. This is a decision tree for 
$n_{\mathcal{C}}$. For $\sigma \in \lk_e(\Delta)$, this follows because $n_{\mathcal{C}/e} = n_{\mathcal{C} - e}$, so $D_{\setminus e}$ is 
a decision tree for $n_{\mathcal{C}/e}$. 
For $\sigma \in \lk_e(\Delta)$, note that $n_{\mathcal{C}}(\sigma) = n_{\mathcal{C}/e}(\sigma)+ 1$.
Thus \[E(D, n_{\mathcal{C}}; q,x) = (1+qx) E(D_{\setminus e}, n_{\mathcal{C} - e}; q,x) = (1+qx)Q(\mathcal{C} - e) = Q(\mathcal{C}) \]

 \item Suppose $e \not\in \Delta_{\mathcal{C}}$. Define $D = N \vee_e D_{\setminus e}$, where $N$ is the tree with one vertex, labeled $N$. This is a decision tree for 
$n_{\mathcal{C}}$.
Thus we have \[E(D, n_{\mathcal{C}}; q,x) = E(D_{\setminus e}, n_{\mathcal{C}-e};q,x) = Q(\mathcal{C}-e) = Q(\mathcal{C}) \]

 \item Suppose $e$ is a coloop. Define $D = D_{\setminus e}$. This is a decision tree for $n_{\mathcal{C}}$, because we never need to query $e$ (since $n(\sigma-e) = n(\sigma+e)$ for all $\sigma \in \Delta$). 
Thus $n_{\mathcal{C}}$ is nonevasive, and so $E(D, n_{\mathcal{C}}; q,x) = 0 = Q(\mathcal{C})$.

 \item Suppose $e$ is not a loop nor coloop, and $e \in \Delta_{\mathcal{C}}$. Define $D = D_{/e} \vee_e D_{\setminus e}$. This is a decision tree for 
$\mathcal{C}$. 
Thus we have \[E(D, n_{\mathcal{C}}; q,x) = qx E(D_{/e}, n_{\mathcal{C}/e}; q,x) + E(D_{\setminus e}, n_{\mathcal{C}-e};q,x) = qQ(\mathcal{C}/e) + Q(\mathcal{C}-e) = Q(\mathcal{C}) \]
\end{enumerate}

Finally, $\Delta$ is semi-nonevasive. In this case, we construct decision trees $D$ such that all the evasive faces are of dimension $r_{\mathcal{C}} - 1$. Then $D$ is optimal, and $\Delta$ is semi-nonevasive.
Let $e \in S$, and by induction assume we have optimal decision trees $D_{\setminus e}$ and $D_{/e}$ for $\del_e(\Delta)$ and $\lk_e(\Delta)$ respectively.
\begin{enumerate}
 \item Suppose $e$ is a loop or coloop. Define $D = D_{ \setminus e}$. This is a decision tree for $\mathcal{C}$, because we never need to query $e$ 
(since $\sigma - e \in \Delta$ if and only if $\sigma + e \in \Delta$ for all $\sigma \subseteq S$). 
Thus $\Delta$ is nonevasive, and so $E(D, r_{\mathcal{C}}; q,x) = 0 = q^{r_{\mathcal{C}}}T_{\mathcal{C}}(0,0)$.

 \item Suppose $e \not\in \Delta_{\mathcal{C}}$. Define $D = N \vee_e D_{\setminus e}$, where $N$ is the tree with one vertex, labeled $N$. This is a decision tree for 
$\Delta$, which has $T_{\mathcal{C} - e}(0,0)$ evasive faces of dimension $r_{\mathcal{C} - e} - 1$. The result follows, since $T_{\mathcal{C}-e} = T_{\mathcal{C}}$.

 \item Suppose $e$ is not a loop nor coloop, and $e \in \Delta_{\mathcal{C}}$. Define $D = D_{/e} \vee_e D_{\setminus e}$. This is a decision tree for 
$\Delta$.
Thus we have $T_{\mathcal{C}/e}(0,0) + T_{\mathcal{C} - e}(0,0) = T_{\mathcal{C}}(0,0)$ evasive faces of dimension $r_{\mathcal{C}} - 1$.
 
\end{enumerate}

\end{proof}

\section{Broken circuits}

In this section, we review the notion of broken circuit complex of a matroid, and extend it to semimatroids. 
The notion of broken circuit complex has appeared in the literature before, in the case when $\mathcal{C}$ is a matroid \cite{brylawski}, or comes from an affine hyperplane arrangement. 
Section 1.5 of \cite{orlik-welker} contains a more thorough review of what is known about the broken circuit complex of an affine hyperplane arrangement.
The purpose will be to give direct combinatorial interpretations of the coefficients of Poincar\'e polynomials for $r$ and $\Delta$. 

Let $\mathcal{C} = (S, \Delta, r)$ be a semimatroid, and fix a linear order on $S$. A \emph{broken circuit} is any face of the form $\sigma - \min \sigma$ where $\sigma$ is a circuit. 
A \emph{no-broken circuit} $\tau$ is a face which does not contain a broken circuit. Clearly, no-broken circuits form a subcomplex $BC(\mathcal{C})$ of $\Delta$, which is called the 
\emph{broken circuit complex}.
Recall that, for a simplicial complex $\Delta$, $f(\Delta, q) = \sum_{i \geq 0} f_i q^i$, where $f_i$ is the number of faces of dimension $i$.

\begin{theorem}
\label{thm:vd}
Let $\mathcal{C}$ be a semimatroid. Then $BC(\mathcal{C})$ is vertex-decomposable. Moreover, $f(BC(\mathcal{C}), q) = q^rT_{\mathcal{C}}(\frac{q+1}{q}, 0)$ and 
$$BC(\mathcal{C}) \simeq \bigvee_{T_{\mathcal{C}}(0,0)} \mathbb{S}^{r_{\mathcal{C}} - 1}$$ 

\end{theorem}
This theorem is already known when $\mathcal{C}$ comes from a matroid (Theorem 3.2.1 in \cite{billera-provan}). 
When $\mathcal{C}$ comes from an affine hyperplane arrangement, our result regarding the homotopy type is Theorem 1.5.6 in \cite{orlik-welker}. 

\begin{proof}
 Let $s = \max S$. If $\{s \} \notin \Delta$, then $BC(\mathcal{C}) = BC(\mathcal{C} - s)$, as no-broken circuits are faces, and the result follows by induction. 
If $s$ is a loop, $BC(\mathcal{C}) = \emptyset$, as every face contains a broken circuit $\emptyset$, and the result follows. So suppose $s$ is a coloop. 
Then $BC(\mathcal{C})$ is the join of $\{\emptyset, \{s\} \}$ and $BC(\mathcal{C} - s)$, as we can add $s$ to any no-broken circuit, and still get a no-broken circuit. 
That is, since $s$ is a coloop, it is not an element of a broken circuit (because if it was, then it would be contained in a loop, and hence would not be a coloop).
By induction, the latter is 
vertex-decomposable, and the join of vertex-decomposable complexes is vertex decomposable. Also, $BC(\mathcal{C})$ is contractible, as $BC(\mathcal{C}) = cone(BC(\mathcal{C} - s))$. 

Finally, suppose $s$ is not a loop or coloop. Then $\del_s(BC(\mathcal{C})) = BC(\mathcal{C} - s)$, and $\lk_s(BC(\mathcal{C})) = BC(\mathcal{C} / s)$, so vertex decomposability follows by induction. 
Both these equalities require $s = \max S$, because this gaurantees that $s$ is the maximum element of any circuit containing $s$, so it is never the minimum. 
Thus, one can check that $\sigma$ is a broken circuit with $s \not\in \sigma$ if and only if $\sigma$ is a broken circuit of $\mathcal{C} - s$. Similarly, $\sigma$ is a broken circuit 
containing $s$ if and only if $\sigma - s$ is a broken circuit of $\mathcal{C} / s$.
\end{proof}

A no-broken circuit $\sigma$ is called \emph{critical} if, for every $x \in \sigma$, there exists $y < x$ such that $\sigma - x + y$ is still a no-broken circuit.
One can observe that $T_{\mathcal{C}}(0,0)$ counts the number of critical no-broken circuits of size $r_{\mathcal{C}}$. 
One can verify the Tutte recursion in this case.
Let $CBC(\mathcal{C})$ denote the set of critical no-broken circuits which are also bases. 
In the literature, these are sometimes referred to as $\beta nbc$ sets. 

Thus, we have combinatorial interpretations of the Betti numbers of $\mathcal{C}$, as well as the relative Betti numbers coming from $r_{\mathcal{C}}$. 
In fact, we can construct a decision tree for $r_{\mathcal{C}}$ whose evasive faces are the no-broken circuits. In particular, $\Delta$ collapses onto $BC(\mathcal{C})$. 
We can also construct a decision tree for $\Delta$ whose evasive faces are the critical no-broken circuits.
\begin{theorem}
\label{thm:collapses}
 Let $\mathcal{C}$ be a semimatroid. Then $\Delta$ collapses onto $BC(\mathcal{C})$, via a sequence of simplicial collapses. 
Moreover, $P(r; q,x) = \sum_{\sigma \in BC(\mathcal{C})} (qx)^{|\sigma|}$, and $P(\Delta; q) = \sum_{\sigma \in CBC(\mathcal{C})} q^{|\sigma|}$.
\end{theorem}

\begin{proof}

Recursively construct a decision tree for $r$ by following the proof of Theorem~\ref{thm:main}, applied to $s = \max S$, rather than an arbitrary element of $S$. 
This slight change in the proof allows us to assume that there exist trees $D_{\setminus s}$ and $D_{/s}$, for $r_{\setminus s}$ and $r_{/s}$, respectively, whose 
evasive faces correspond to elements of $BC(\mathcal{C} \setminus s)$ and $BC(\mathcal{C}/s)$. Then we construct a new tree $D = D^+_{/s} \vee_s D_{\setminus s}$. 
A careful examination of the proof of Theorem~\ref{thm:vd} reveals then that 
the evasive faces of $D$ are the faces of $BC(\mathcal{C})$. 

Similarly, if we construct a decision tree for $\Delta$ by following the proof of Theorem~\ref{thm:main}, applied to $s = \max S$, then we may assume by induction 
that the evasive faces of $D_{\setminus s}$ are the critical no-broken circuit bases of $\mathcal{C} - s$, and the evasive faces of $D_{/ s}$ are the critical no-broken circuit bases of 
$\mathcal{C}/s$.

\end{proof}

\section{Strong pseudo-independence complexes}
Jonsson defined the notion of strong pseudo-independence complexes of matroids, to capture some of the combinatorial properties complexes like $Bip(G)$ possessed. 

We can restate Jonsson's result as follows:
\begin{theorem}
 Let $M$ be a matroid with rank function $r$, and assume $\Delta$ is a strong pseudo-independence complex over $M$. Then 
\begin{itemize}
 \item $\Delta$ is homology decidable,
 \item the $(r-2)$-skeleton of $\Delta$ is vertex decomposable,
 \item $\Delta$ is homotopy-equivalent to a wedge of $(r-2)$-dimensional spheres.
\end{itemize}
\end{theorem}
It turns out that strong pseudo-independence complexes, and semimatroids are equivalent notions.
Thus, this present paper supplies at least two new extensions to Jonsson's Theorem. First, $\Delta$ collapses onto a 
vertex-decomposable subcomplex, as a result of Theorems \ref{thm:collapses} and \ref{thm:vd}. 
Second, the number of spheres in the wedge is given by an evaluation of the Tutte polynomial.

Let $M$ be a matroid with ground set $S$, and rank function $r$, and $\Delta$ be a simplicial complex over $S$. Then $\Delta$ is a \emph{pseudo-independence complex} if, 
whenever $\sigma \in \Delta$, $x \in E \setminus \sigma$, and $r(\sigma + x) > r(\sigma)$, then $\sigma + x \in \Delta$. $\Delta$ is a \emph{strong} complex over $M$ if, 
whenever $\sigma \in \Delta$, and there exists $x \in S \setminus \sigma$ such that $r(\sigma + x) = r(\sigma)$ and $\sigma + x \in \Delta$, then $x$ is a cone point of $\lk_{\sigma}(\Delta)$.
Given a semimatroid $\mathcal{C} = (S, \Delta, r)$, let $M$ be the matroid on $S$ with rank function $r'(\sigma) = \max \{ r(\tau): \tau \subseteq \sigma, \tau \in \mathcal{C} \}$.
Then we have the following:
\begin{theorem}
Given a semimatroid $\mathcal{C} = (S, \Delta, r)$, define $M$ to be the matroid on $S$ with rank function $r'(\sigma) = \max \{ r(\tau): \tau \subseteq \sigma, \tau \in \mathcal{C} \}$.
Then $\Delta$ is a strong pseudo-independence complex over $M$. 
\end{theorem}

\begin{proof}
Let $\sigma \in \Delta$, $x \in S \setminus \sigma$, and suppose $r'(\sigma + x) > r'(\sigma) = r(\sigma)$. 
Hence there is a set $\tau \subset \sigma$ such that $\tau + x \in \Delta$, and $r(\tau+x) = r(\sigma) + 1$. Then there exists $t \in (\tau + x) \setminus \sigma$ with $\sigma + t \in \Delta$. 
Since $\tau + x \setminus \sigma = \{x \}$, it follows that $t = x$, and so $\sigma + x \in \Delta$. Hence, $\Delta$ is a pseudo-independence complex.

$\Delta$ is a strong pseudo-independence complex. Let $\sigma \in \Delta$, $x \in S \setminus \sigma$, and suppose $r'(\sigma + x) = r'(\sigma)$, and $\sigma + x \in \Delta$.
We must show that, for every $\sigma \subseteq \tau \in \Delta$, if $x \not\in \tau$, then $\tau + x \in \Delta$. 
Observe that $r((\sigma+x) \cap \tau) = r(\sigma) = r(\sigma+x)$, so $(\sigma+x) \cup \tau = \tau+x \in \Delta$, since $\Delta$ is a semimatroid. 
Thus $x$ is a cone point, and $\Delta$ is a strong pseudo-independence complex.
\end{proof}

The converse is also true:
\begin{theorem}
Let $\Delta$ be a strong pseudo-independence complex over a matroid $M$ with ground set $S$. Then $(S, \Delta, r|_{\Delta})$ is a semimatroid.
\end{theorem}
\begin{proof}
The first three conditions in the definition of semimatroid are automatically satisfied, since $r$ is the rank function of a matroid. So suppose $\sigma, \tau \in \Delta$ with 
$r(\sigma) = r(\sigma \cap \tau)$. We prove $\sigma \cup \tau \in \Delta$ by induction on $\sigma \setminus \tau$. Consider $x \in \sigma \setminus \tau$, and let $\rho = \sigma \cap \tau$. 
Since $r(\rho) \leq r(\rho+x) \leq r(\sigma) = r(\rho)$, and $x \in \sigma \in \Delta$, it follows that $x$ is a cone point of $\lk_{\rho}(\Delta)$.
 Since $\tau \setminus \rho \in \lk_{\rho}(\Delta)$, it follows that $\tau + x \in \Delta$. Applying induction to $\sigma, \tau + x$, it follows that $\sigma \cup \tau \in \Delta$.

Let $\sigma, \tau \in \Delta$ such that $r(\sigma) < r(\tau)$. Since $r$ is the rank function of a matroid, there exists $x \in \tau \setminus \sigma$ such that $r(\sigma) < r(\sigma + x)$. Since $\Delta$ is a 
pseudo-independence complex, $\sigma + x \in \Delta$. Thus, $(S, \Delta, r|_{\Delta})$ is a semimatroid.
\end{proof}

We end this section be discussing dual strong pseudo-independnce complexes. Jonsson defines a complex $\Delta$ to be SPI$^{\ast}$ over a matroid $M$ if and only if $\Delta^{\ast}$, the Alexander dual of $\Delta$, 
is a strong pseudo-independence complex over $M^{\ast}$, the dual of $M$. Ardila showed that the Alexander dual of a semimatroid is again a semimatroid (Proposition 7.2 in \cite{ardila}); hence, SPI$^{\ast}$ complexes are also semimatroids. 

\section{Discrete Morse theory for monotone functions}
\label{sec:discmorse}

In this section, we extend some of the main theorems of discrete Morse theory to monotone functions. To keep with the theme of the paper, we will state the results in terms of decision trees. 
The proofs will essentially follow the proof of Theorem 11.13 in Kozlov \cite{kozlov}, only with slight modifications to deal with monotone functions. We only include it for the 
sake of completeness.

The usual combinatorial approach to discrete Morse theory is via acyclic matchings; this approach is due to Chari \cite{chari}.
 However, it is known that there is a relationship between acyclic matchings and decision trees: given any acyclic matching, one can construct a decision tree, and vice-versa (a proof is given in \cite{jonsson-decision}). 
This implies that there is a relationship between discrete Morse theory and decision trees; this section is meant to clarify this relationship, by showing how to go directly from 
a decision tree to a sequence of simplicial collapses.

\begin{definition}
 Let $f: \Delta \to P$ be a monotone function. An $f$-tree $T$ is a set decision tree if it satisfies exactly one of the following:
\begin{enumerate}
 \item $T$ consists of a single vertex, labeled $N$, and $\Delta$ is void.
 \item $T$ consists of only a single vertex, labeled $\hat{0}$, $f$ is the constant function of value $\hat{0}$, and $\Delta$ is a simplex.
 \item The root of $T$ is labeled by $\sigma \subset S$, with $\sigma \in \Delta$, the left subtree $T_{\ell}$ is an element decision tree for $f_{\setminus \sigma}$, 
and the right subtree $T_{r}$ is an element decision tree for $f_{/ \sigma}$. 
\end{enumerate}
\end{definition}

Given a set decision tree $T$ for $f$, let $L$ denote the leaves of $L_T$. When we run the decision algorithm for a given $\sigma \subseteq S$, we land on some leave $\ell(\sigma)$. 
This defines a function $\ell: 2^S \to L_T$. This function has some nice properties.
\begin{lemma}
Let $T$ be a decision tree for a monotone function $f$, and let $v \in L_T$. Then $\ell^{-1}(v)$ is a boolean interval. 
Moreover, the faces of $\Delta$ are partitioned into these boolean intervals.
\end{lemma}
\begin{proof}
Proof by induction: if $T$ consists of only one vertex, then $\Delta$ is a simplex, $\ell^{-1}(v) = \Delta$, so we are done. 
Otherwise, let $T_{\ell}$ be the left subtree of $T$, and $T_r$ be the right subtree of $T$. every leaf is a leaf of one of these two subtrees, 
so by induction $\ell^{-1}(v)$ must be a boolean interval. To show that the boolean intervals partition $\Delta$, let $\sigma$ be the label of the root 
of $T$. Let $\tau$ be a face. If $\sigma \subset \tau$, then $\sigma$ is contained in one of the boolean intervals of a leaf of $T_{\ell}$. Otherwise, 
$\sigma$ is a face contained in one of the boolean intervals of a leaf of $T_r$.
\end{proof}

A face $\sigma \in \Delta$ is $p$-evasive if $f(\sigma) = p$, and $\ell^{-1}(\ell(\sigma)) = \{\sigma \}$. The reader can check that, when $T$ is an element decision tree, this definition matches the earlier one.

\begin{theorem}
\label{thm:fundthm}
Let $f: \Delta \to P$ be a monotone function, and let $T$ be a set decision tree, such that $\emptyset$ is nonevasive. Then there exists a cell complex $X$, a homotopy equivalence 
$\varphi: \Delta \to X$ and a monotone function $f: X \to P$, such that:
\begin{enumerate}
\item The $i$-cells of $X_{\leq p} \setminus X_{< p}$ are indexed by the elements of $Ev_{T,p,i}$.
\item $\varphi$ restricts to a homotopy equivalence between $(\Delta_{\leq p}, \Delta_{< p})$ and $(X_{\leq p}, X_{< p})$ for all $p \in P$.
\item For $\sigma \in E_{T, p, i}$, the corresponding cell in $X$ also has function value $p$.
\item If $\cup_{q \leq p, i \in \mathbb{N}} Ev_{T, q, i}$ is a simplicial complex for all $p \in P$, then $X$ is also a simplicial complex, and $\Delta$ collapses onto $X$ via a sequence of simplicial collapses.

\end{enumerate}

\end{theorem}
There is a total order on $L_T$, which we describe. Given two leaves $u$ and $v$, consider the shortest paths $P$ and $Q$ from the root to $u$ and $v$, respectively, 
and let $x$ be the last node $P$ and $Q$ have in common. Then $u < v$ if $u$ is a descendant of the left child of $x$, and $v$ is a descendant of the right child of $x$. 
In view of the above proposition, this means that we actually have partitioned $\Delta$ into a sequence of boolean intervals, where $f$ is constant on each interval. 
Thus we obtain a sequence of homotopy equivalences arising from cellular collapses. 

\begin{remark}
 Theorem \ref{thm:fundthm} is still true when $\emptyset$ is evasive, provided we add one more $0$-cell and $(-1)$-cell to $X$. The $(-1)$-cell has function value $f(\emptyset)$, 
and the $0$-cell has function value given by the label of the largest element of $L_T$ whose label is not $N$.
\end{remark}

The proof relies on the following technical lemma, which is well-known.
\begin{lemma}
\label{lem:technical}
Let $\tau$ be a simplex, and $\sigma \subset \tau$ be a non-empty face, $\sigma \neq \tau$. Then there exists a sequence of simplicial collapses from $\tau$ to $\tau \setminus \sigma$. 
This gives rise to a homotopy equivalence $\varphi: \tau \to \tau \setminus \sigma$.
\end{lemma}

\begin{proof}
 Let $x \in \tau \setminus \sigma$, $\Delta = \del_{x}(\lk_{\sigma}(\tau))$. Order the faces of $\Delta$ linearly so that $\gamma < \rho$ implies that $\dim \gamma \geq \dim \rho$, for all $\gamma, \rho \in \Delta$.
Let $\gamma_0, \gamma_1, \ldots, \gamma_k$ denote the resulting total order. Then let $(\gamma_0 \cup \sigma, \gamma_0, \cup \sigma \cup \{x \}), \ldots, (\gamma_k \cup \sigma, \gamma_k \cup \sigma \cup \{x\})$ 
denote the resulting sequence of collapses. Since we order by decreasing dimension, at the $i$th step $\gamma_i \cup \sigma$ is a free face, so we can perform the simplicial collapse. 
After the collapses are performed, the only faces left are those of $\tau \setminus \sigma$.
\end{proof}

\begin{proof}[Proof of Theorem~\ref{thm:fundthm}]

Let $f: \Delta \to P$ be monotone, and $T$ be a set decision tree. The result is proven by induction on the number of leaves in $L_T$ whose label is not $N$. 
If only one leaf is not labeled $N$, then $\Delta$ is a simplex, and $f$ is a constant function. If $\Delta$ is the empty simplex, then the empty set is 
evasive. Otherwise, $\Delta$ collapses to a point $x$ via some sequence of simplicial collapses; the Theorem follows with $X = \{x \}$, and 
$\varphi$ the homotopy equivalence induced from the sequence of collapses.

Given the ordering on $L_T$, consider the first vertex $v$ whose label is not $N$. 
Then $\ell^{-1}(v)$ is non-empty, and is a boolean interval $[\sigma, \tau]$, for some faces $\sigma$ and $\tau$. 
Moreover, $\sigma \neq \emptyset$.

We construct a decision tree for $f_{\setminus \sigma}$. Replace the label of $v$ with $N$, to obtain a new tree $T'$. Let $w$ be the parent of $v$.
The result is almost a decision tree for $f_{\setminus \sigma}$. There may be an issue where two leaves with the same parent both have the label $N$ 
(which would not satisfy our definition of decision tree). If this is the case, remove those leaves, and replace the parent node's label with $N$. 
Continue doing so until there are no longer siblings who are both labeled $N$. The result is a decision tree for $f_{\setminus \sigma}$.
Thus, by induction, there exists $X'$, and a homotopy equivalence $\varphi$, satisfying 1-3 of the Theorem, with respect to the monotone function $f_{\setminus \sigma}$. 

Suppose $\sigma \neq \tau$. We claim that $\sigma$ is a free face. If $\sigma \subset \rho$, then $\ell(\rho) \leq \ell(\sigma)$. If $\ell(\rho) < \ell(\sigma)$, 
then the label of $\ell(\rho)$ is $N$, which implies that $\rho \not\in\Delta$. Thus, $\ell(\rho) = \ell(\sigma)$, and $\sigma$ is a free face. 
Then, by Lemma~\ref{lem:technical}, $\Delta$ collapses onto $\del_{\sigma}(\Delta)$, which gives a homotopy equivalence $\psi$. 
Then the theorem holds with $X = X'$, and homotopy equivalence $\varphi \circ \psi$.

Suppose $\sigma = \tau$. Then $\sigma$ is evasive. Let $X = X' \cup_{\varphi(\del \sigma)} \sigma$. Then there is a homotopy equivalence between $\Delta$ and $X$. 
We need to show that $f(\sigma) \geq f(\rho)$ for any cell $\rho$ in $X$ which lies on $\sigma$. Note that $\varphi$ is constructed inductively through a series of simplicial collapses. 
For any $\pi$ that is a face of $\sigma$ in $\Delta$, we see that each collapse either does not alter 
$\pi$, or replaces $\pi$ with a union of faces $\omega$ with $f(\omega) \leq f(\pi) \leq f(\sigma)$. Thus, after 
all the collapses are done, every face $\rho$ of $\sigma$ in $X$ satisfies $f(\rho) \leq f(\sigma)$. By induction, parts 1-3 of the theorem hold.

To show part 4, note that the linear order on $L_T$ induces a linear order on $NE_T$, the subset of $L_T$ of vertices $v$ for which $\ell^{-1}(v)$ has more than one element.
 Using Lemma~\ref{lem:technical}, we obtain a sequence of collapses from $\Delta$ to the subcomplex generated by the evasive faces.

\end{proof}
Now we prove the lemmas from Section 3:

\begin{proof}[Proof of Lemma~\ref{lem:morseinequalities}]
Let $T$ be a decision tree for a monotone function $f: \Delta \to P$. 
By Theorem~\ref{thm:fundthm}, there exists a cell complex $X$ and a homotopy equivalence $\varphi: \Delta \to X$ satisfying the conditions 1-4. Fix $p \in P$. 
Then $(\Delta_{\leq p}, \Delta_{< p})$ is homotopy equivalent to $(X_{\leq p}, X_{< p})$ by part 3. Hence $\widetilde{\beta}_i(\Delta_{\leq p}, \Delta_{< p}) \leq |C_i(X_{\leq p}, X_{< p})|$, 
the number of $i$-dimensional cells $\sigma$ in $X$ with $f(\sigma) = p$. However, by part 2, this value is $ev_T(f,p,i)$, and the result follows.

\end{proof}

\begin{proof}[Proof of Lemma~\ref{lem:simple}]
Let $T$ be a decision tree with the property that all evasive $p$-faces are equidimensional. This means that, for each $p \in P$, there is an integer $n_p$ such that $|E_{T, p, i}| = 0$ unless $i = n_p$. 
Let $X$ be the cell complex given by Theorem~\ref{thm:fundthm}. Then we see that the relative chain groups $C_i(X_{\leq p}, X_{< p})$ are only non-trivial when $i = n_p$. Hence, the 
relative homology is concentrated in dimension $n_p$, and in fact $\widetilde{\beta}_{i}(X_{\leq p}, X_{< p}) = \dim C_{i}(X_{\leq p}, X_{< p}) = |Ev_{T,p, i}|$. Thus $T$ is optimal, and $f$ is semi-nonevasive.
\end{proof}

To conclude this section, we make some remarks about discrete Morse theory via decision trees. Sometimes this approach is easier than the traditional acyclic matching approach. 
For instance, we were able to use the deletion-contraction recurrence for semimatroids to construct decision trees in this current paper. 
However, there are some drawbacks to this approach: decision trees are only defined for abstract simplicial complexes with a finite vertex set. 
Also, the Fundamental Theorem of Discrete Morse Theory contains a statement which gives a formula for the boundary operator of $X$ in terms of the boundary operator of $\Delta$.
However, the boundary operator is not an invariant of the decision tree, because it depends on the collapses performed in Lemma~\ref{lem:technical}.
Also, if one is interested in the cell complex obtained from the Fundamental Theorem \ref{thm:fundthm}, then in order to describe the boundary operator one needs to 
keep track of the collapses, and it appears that approaching discrete Morse theory via acyclic matchings \cite{chari}.

\section{Acknowledgments}
The author would like to thank John Shareshian, Jakob Jonsson, Matthias Beck, Volkmar Welker, and Chris Severs for many stimulating conversations related to this work, 
and for help with finding references in the literature. The author was partially supported by an NSF grant DMS-0932078,
administered by the Mathematical Sciences Research Institute while the
author was in residence at MSRI during the Complementary Program, Spring 2011.
This work began during the
visit of the author to MSRI and we thank the institute for its
hospitality.

\bibliographystyle{amsalpha}
\bibliography{k_parbib.bib}

\end{document}